\documentclass[12pt,final,oneside]{amsart}

\usepackage{amssymb, amsmath, amsthm,geometry}
\usepackage[lithuanian,english]{babel}
\usepackage{mathrsfs} 

\usepackage{geometry}

\usepackage{bm}
\usepackage{subfig,tikz}
\usepackage{wrapfig}
\usepackage{color}
\usepackage{xspace}
\usetikzlibrary{shapes.geometric}

\definecolor{darkblue}{rgb}{0.0, 0.0, 0.55}

\usepackage{graphicx}
\usepackage[colorlinks=true, citecolor=black, linkcolor=black, urlcolor=black]{hyperref}
 \usepackage{xcolor}
\usepackage{hyperref}

\usepackage{lineno}

\usepackage{ifdraft}
\ifoptionfinal{
\usepackage[disable]{todonotes}
}
{
\usepackage[norefs, nocites]{refcheck}

\usepackage[notref, notcite]{showkeys}
\usepackage[bordercolor=white, color=white]{todonotes}
}

\newtheorem{lemma}{Lemma}
\newtheorem{theorem}{Theorem}
\newtheorem{corollary}{Corollary}

\theoremstyle{remark}
\newtheorem{example}{Example} 
\newtheorem{remark}{Remark} 
\theoremstyle{theorem}
\newtheorem{hypothesis}{Hypothesis}

\def\R{\mathbb R}

\def\N{\mathbb N}

\def\g{\bar{g}}
\def\M{\mathcal M}

\def\p{\partial}

\DeclareMathOperator{\supp}{supp}
\DeclareMathOperator{\dist}{dist}

\title[The Light ray transform in Lorentzian geometries]{The Light ray transform in Stationary and Static Lorentzian geometries}
\author[Feizmohammadi]{Ali Feizmohammadi}
\address{Department of Mathematics, University College London, 
Gower Street, London UK, WC1E 6BT.}
\email{a.feizmohammadi@ucl.ac.uk}
\author[Ilmavirta]{Joonas Ilmavirta}
\address{Department of Mathematics and Statistics, University of Jyv\"{a}skyl\"{a}, P.O. Box 35 (MaD), Finland}
\email{joonas.ilmavirta@jyu.fi}
\author[Oksanen]{Lauri Oksanen}
\address{Department of Mathematics, University College London, 
Gower Street, London UK, WC1E 6BT.}
\email{l.oksanen@ucl.ac.uk}

\begin{document}
\begin{abstract}
Given a Lorentzian manifold, the light ray transform of a function is its integrals along null geodesics. This paper is concerned with the injectivity of the light ray transform on functions and tensors, up to the natural gauge for the problem. First, we study the injectivity of the light ray transform of a scalar function on a globally hyperbolic stationary Lorentzian manifold and prove injectivity holds if either a convex foliation condition is satisfied on a Cauchy surface on the manifold or the manifold is real analytic and null geodesics do not have cut points. Next, we consider the light ray transform on tensor fields of arbitrary rank in the more restrictive class of static Lorentzian manifolds and show that if the geodesic ray transform on tensors defined on the spatial part of the manifold is injective up to the natural gauge, then the light ray transform on tensors is also injective up to its natural gauge. Finally, we provide applications of our results to some inverse problems about recovery of coefficients for hyperbolic partial differential equations from boundary data.
\end{abstract}
\maketitle

\section{Introduction}
\label{intro}
Let $(\mathcal N,\g)$ be a smooth globally hyperbolic Lorentzian manifold of dimension $1+n$ with signature $(-,+,\ldots,+)$. For a review of the notion of global hyperbolicity we refer the reader to \cite[Chapter 14]{On}. For the purposes of this paper, it suffices to recall that such manifolds have a Cauchy hypersurface, that is a hypersurface which is intersected by any causal curve exactly once. We are interested in studying the injectivity of the so called light ray transform on functions and tensors over such Lorentzian manifolds.  

To formulate the problem precisely we introduce some notations. For each $m=0,1,\ldots$, let $\mathcal S^m=\mathcal S^m(\mathcal N)$ denote the vector bundle of symmetric tensors of rank $m$ on $\mathcal N$. In local coordinates each $\alpha \in C^{\infty}_c(\mathcal N;\mathcal S^m)$ can be written as 
$$
\alpha(y,w) = \alpha_{j_1 \dots j_m}(y) w^{j_1} \dots w^{j_m}, \quad\forall\, (y,w) \in T\mathcal N,
$$
where we are using the Einstein summation convention. Next, let $\beta$ be a maximal (that is inextensible) null geodesic in $(\mathcal N,\g)$, namely a geodesic whose tangent vector at each point is light-like: 
\begin{align}\label{affine} \nabla^{\g}_{\dot{\beta}(s)}\dot{\beta}(s)=0, \quad \text{and}\quad \g(\dot{\beta}(s),\dot{\beta}(s))=0.\end{align}
Observe that equation \eqref{affine} defines the parametrization of $\beta$ uniquely up to a group of affine re-parametrizations. Given any choice of such parametrization, we define the light ray transform of $\alpha \in  C^{\infty}_c(\mathcal N;\mathcal S^m)$ along $\beta$ as follows:
\begin{align}\label{lightraydef} \mathcal L_\beta \alpha = \int_\R \alpha(\beta(s),\dot{\beta}(s))\,ds.\end{align}
Note that, by global hyperbolicity the null geodesic $\beta(s)$ will lie outside of any compact set $K\subset \mathcal N$, for $|s|$ large enough (see \cite[Lemma 13, p. 408]{On}), and therefore the integral in \eqref{lightraydef} is well-defined for compactly supported $\alpha$. Note also that the domain of integration in \eqref{lightraydef} is also justified even when $\beta$ is not complete since $\alpha$ is compactly supported. Finally, observe that an affine reparametrization of $\beta$ results in the integral \eqref{lightraydef} to be scaled. Together with the linearity of the map $\mathcal L$, this implies that the choice of the parametrization is of no significance provided that we are concerned with injectivity of the light ray transform on $\mathcal N$.


\subsection{The case of stationary geometries}
\label{stationarycase}

The first result in our paper is concerned with the injectivity of the light ray transform on scalar functions under the additional assumption that $(\mathcal N,\g)$ is {\em stationary}, in the sense that there exists a smooth complete time-like Killing vector field. Let $N \subset \mathcal N$ denote a fixed Cauchy hypersurface in $\mathcal N$, write $g=\bar{g}|_{N}$ and observe that $(N,g)$ is a Riemannian manifold. It is well-known (for example \cite[Lemma 3.3]{JS}) that the manifold $(\mathcal N,\g)$ admits a canonical isometric embedding $\Phi: \R\times N \to \mathcal N$ with
\begin{align}\label{stationary} \Phi^*\g= -(\kappa-|\eta|_g^2)\,dt^2+dt\otimes \eta+\eta\otimes dt+g,\end{align}
where $\kappa$ is a smooth positive function on $N$ and $\eta$ is a smooth co-vector field on $N$. For the convenience of the reader we show this in Section~\ref{stationary_structure}. In the more restrictive case where additionally the one-form $\eta$ in \eqref{stationary} vanishes identically in $N$, we call the manifold $\mathcal N$ to be {\em static}. 

In the setting of stationary geometries introduced above, we prove injectivity of the light ray transform on scalar functions under either one of the two hypotheses that we will formulate next. To simplify the statement of these hypotheses, we only consider injectivity of the light ray transform among functions that are compactly supported in a submanifold $\M \subset \mathcal N$, given by $\M=\Phi(\R\times M)$, where $M\subset N$ is a compact submanifold of dimension $n$ and with a smooth boundary $\p M$. Consequently, we are studying the injectivity of the light ray transform on the Lorentzian manifold $\M$.

To state the first hypothesis, we recall some concepts from Lorentzian geometry, namely, the notion of {\em time-separation} and {\em null cut locus}. The time-seperation function, $\tau(p,q)$, between two points $p$ and $q$ is defined as the supremum of the semi-Riemannian length of all future pointing causal (non-spacelike) curves connecting $p$ to $q$ and zero if there is no such path. Next, let $p \in \mathcal M$ and $\beta:I\to \M$ denote a future pointing null geodesic with $\beta(0)=p$ and set 
$$s_0=\sup\{s \in I \,|\, \tau(p,\beta(s))=0\}.$$ If $s_0 \in I^{\text{int}}$, we call $\beta(s_0)$ the future null cut-point of $p$ along $\beta$ (see \cite[Section 5]{BE}). The null cut locus $C_N^+(p)$ is then defined as the set of all of all future null cut points of $p$.
   
\begin{hypothesis}
\label{hypo0}
The Lorentzian manifold $(\mathcal N,\g)$, the Cauchy hypersurface $N$ and the Killing vector field are real analytic. Moreover, $C^+_N(p)=\emptyset$ for all $p \in \M$.  
\end{hypothesis}

Before stating the second hypothesis, we need to make more definitions. We introduce the conformally scaled metric
\begin{align}\label{conf} \Phi^*\g_c= -dt^2+dt\otimes \eta_c+\eta_c\otimes dt+g_c,\end{align} 
where $\eta_c=c\eta$, $g_c=cg$ and $c=\frac{1}{\kappa-|\eta|_g^2}.$ 
Next, we define $\mathscr G$ to be the set of smooth curves $b$ on $M$ that satisfy the following ordinary differential equation:
\begin{align}
\label{generalcurve}
\nabla^{g_c}_{\dot b}\dot{b}=G(b,\dot b),
\end{align}
subject to the initial data $(b(0),\dot b(0))\in TM$. The function $G(z,v)$ is defined for each $(z,v) \in TM$ as follows: 
\begin{align}
\label{G_def}
G(z,v) = -(\frac{\kappa-|\eta|_{g}^2}{\kappa})\left((\nabla^{g_c}_{v}\eta_c)\,v\right)\,\eta_c^{\sharp}-(\eta_c v + \sqrt{(\eta_c v)^2 + |v|_{g_c}^2}) \, F(z,v).
\end{align}
Here $\eta_c^\sharp$ is the vector dual to $\eta_c$ with respect to $g_c$ and the terms $(\nabla^{g_c}_{v} \eta_c)\,v$ and $\eta_c v$ denote the natural pairing between the one forms $\nabla^{g_c}_{v} \eta_c$ and $\eta_c$ with the vector $v$ respectively. All the terms in \eqref{G_def} are evaluated at the point $z \in M$. Finally, the term $F(z,v)$ is the vector field defined through
$$ F(z,v) = d\eta_c(\cdot,v)^{\sharp} - \left(\frac{\kappa-|\eta|_{g}^2}{\kappa}\right) d\eta_c(\eta_c^{\sharp},v)\, \eta_c^{\sharp}.$$ 

The second hypothesis relies on a notion of foliation by a family of {\em strictly convex} hypersurfaces with respect to curves in $\mathscr G$, and can be stated as follows.
\begin{hypothesis}
\label{hypo1}
The dimension $n$ of $M$ satisfies $n \geq 3$, and there is a function $\rho:M\to [0,l]$, so that the following conditions hold: 
\begin{itemize}
\item[(i)]{$d\rho \neq 0$ when $\rho>0$, $\rho^{-1}(l)=\p M$ and $\rho^{-1}(0)$ has empty interior.}
\item[(ii)]{For any $b \in \mathscr G$, if $\frac{d}{dt} \rho(b(t))=0$, then $\frac{d^2}{dt^2}\rho(b(t))>0$.}
\end{itemize}
\end{hypothesis}
In Section~\ref{thm1section}, we provide an example of such manifolds. The main theorem can now be stated as follows.
\begin{theorem}
\label{t1}
Let $(\mathcal N,\g)$ be a stationary globally hyperbolic Lorentzian manifold of dimension $1+n$. Let $\Phi$ be an embedding satisfying \eqref{stationary} and let $\M=\Phi(\R\times M)$ where $M$ is a compact $n$ dimensional submanifold of $N$ with smooth boundary such that Hypothesis~\ref{hypo0} or Hypothesis~\ref{hypo1} holds. Then the light ray transform in $(\M,\g)$ is injective on scalar functions. In other words, given any $f \in C^\infty_c(\M)$, there holds: 
$$ \mathcal L_\beta\,f =0 \quad \text{for all maximal $\beta$ in $\M$} \implies f \equiv 0.$$
\end{theorem}
\bigskip
Although Hypotheses~\ref{hypo0}--\ref{hypo1} and Theorem~\ref{t1} are stated on spatially compact submanifolds $\M$ of $\mathcal N$, we can immediately obtain the following global corollary.
\begin{corollary}
\label{cor1}
Let $(\mathcal N,\g)$ denote a globally hyperbolic stationary Lorentzian manifold of dimension $1+n$ with $n\geq 3$, such that there exists a non-compact Cauchy hypersurface $N$. Let $\Phi$ be an embedding satisfying \eqref{stationary} and suppose that there exists a function $\rho:N\to [0,\infty)$, such that Hypothesis~\ref{hypo1} holds on each $M_l=\{\rho\leq l\}$ with respect to the function $\rho|_{M_l}$. Then the light ray transform on $(\mathcal N,\g)$ is injective on scalar functions.  
\end{corollary}
Indeed, note that due to the non-compactness assumption on $N$, given any scalar function $f$ on $\mathcal N$ with compact support, there exists a large enough $l$ such that $\supp f \subset \Phi(\R\times M)$ with $M=\{\rho\leq l\}$. The corollary follows since $M$ satisfies Hypothesis~\ref{hypo1} with $\rho|_M$.

Theorem~\ref{t1} has applications to the recovery of zeroth order time-dependent coefficients for the wave equation from boundary data. More specifically, consider the following initial boundary value problem on $(\M,\g)$:
\begin{align}
\label{eq1}
\left\{ \begin{array}{ll} \Box_{\g} u + q \,u  =  0, & \mbox{on}\ \M,\\  
u  =  h, & \mbox{on}\ \p\M,\\
u = 0, & \mbox{on}\ \Phi((-\infty,0)\times M).\\
\end{array} \right.
\end{align}
where $\Box_{\g}$ denotes the wave operator on $(\M,\g)$ given in local coordinates by the expression 
$$ \Box_{\g} u = -\sum_{i,j=0}^n|\det\g\,|^{-\frac{1}{2}}\frac{\p}{\p x^i}\left(|\det\g\,|^{\frac{1}{2}}\g^{ij}\frac{\p}{\p x^j}u\right)$$
and $q$ is a smooth a priori unknown function with compact support in the set $\Phi((0,\infty)\times M)$. We consider the problem of recovering $q$ from the Dirichlet to Neumann operator $\Lambda_q$ that is defined for all $h$ compactly supported in $\p \M$ through 
\[\Lambda_{q}:h\mapsto \partial_{\bar{\nu}} u|_{\partial \M}.\]
It can be shown that the question of unique recovery of $q$ from $\Lambda_q$ reduces to the question of injectivity of the light ray transform on $(\M,\g)$ (see for example \cite{SY}). As an immediate corollary of Theorem~\ref{t1}, we deduce that $\Lambda_q$ determines $q$ uniquely, if Hypothesis~\ref{hypo0} or Hypothesis~\ref{hypo1} holds.

\subsection{The case of static geometries}
\label{staticcase}
Given a static globally hyperbolic Lorentzian manifold there exists an embedding $\Phi:\R\times N \to \mathcal N$ such that \eqref{stationary} holds with $\eta \equiv 0$. Analogously to the previous section, we define $\M= \Phi(\R \times M)$ where $M \subset N$ is a compact manifold of dimension $n$ with smooth boundary and study the injectivity of the light ray transform on tensors of arbitrary rank $m$ over $\M$. 

Before presenting the main result, we need to recall the definition of the geodesic ray transform on tensors in $(M,g_c)$. To this end, suppose that $\gamma$ is a unit speed geodesic in $(M,g_c)$. We define the bundle
$$ \p_{\text{in}}SM = \{ (x,v)\in TM\,|\, x \in \p M,\, v \in T_xM,\, |v|_{g_c}=1,\, g_c(v,\nu)<0\},$$
where $\nu$ denotes the unit outward pointing normal vector to $\p M$ at the point $x$. For each $(x,v) \in \p_{\text{in}}SM$, we consider the unique geodesic $\gamma$ with initial data $(x,v)$ and define
$$ \tau_+(x,v)= \inf \{r>0\,|\, \gamma(r;x,v) \in \p M, \, \dot{\gamma}(r;x,v) \notin T_{\gamma(r;x,v)}\p M\}.$$
We assume that the manifold $(M,g_c)$ is non-trapping, that is, for all unit speed geodesics $\gamma(\cdot;x,v)$ with $(x,v)\in \p_{\text{in}}SM$, there holds $\tau_+(x,v)<\infty$. Finally, let $S^m=S^m(M)$ denote the bundle of symmetric tensors of rank $m$ on $M$ (not to be confused with $\mathcal S^m$, the corresponding bundle on $\mathcal N$) and define the geodesic ray transform of $\omega \in C^{\infty}_c(M;S^m)$ along $\gamma$ in $M$ as follows:
\begin{align}
\label{geodesicraydef}
\mathcal I \omega(x,v):= \int_0^{\tau_+(x,v)} \omega(\gamma(\tau;x,v),\dot{\gamma}(\tau;x,v))\,d\tau.
\end{align}
Here, analogously to the Lorentzian case, we have in local coordinates
$$
\omega(y,w) = \omega_{j_1 \dots j_m}(y) w^{j_1} \dots w^{j_m}, \quad \forall\, (y,w) \in TM.
$$
We require the following hypothesis to hold:
\begin{hypothesis}
\label{hypo2}
The geodesic ray transform on $(M,g_c)$ is solenoidally injective. In other words, for any $\omega \in \mathcal C^{\infty}_c(M;S^m)$, there holds:
\[ \mathcal I \omega (x,v)=0\quad  \forall \, (x,v) \in \p_{\text{in}}SM \implies \exists\, \theta\quad \text{such that}\quad \omega=d^s\theta,\quad \theta|_{\p M}=0,\]
where $d^s$ denotes the symmetrized covariant derivative on $(M,g_c)$. 
\end{hypothesis}

The study of the solenoidal injectivity of the geodesic ray transform on tensors of arbitrary rank has a rich literature. For example, Hypothesis~\ref{hypo2} with a fixed $m=0,1$ is known to be true when $(M,g_c)$ is a simple manifold \cite{Mu1,Mu2,AR} or has strictly convex boundary and admits a foliation by strictly convex hypersurfaces \cite{UV}. Under the latter condition it was later proved that Hypothesis~\ref{hypo2} holds for all $m=0,1,2$ \cite{SUV}, and subsequently that it holds for all $m=0,1,2,\ldots$ \cite{dHUZ}. For more related results we refer the reader to \cite{CM,PSU1,PSU2,PS,SU} and the review article \cite{IM}. We can now state our main theorem for the injectivity of the light ray transform on tensors.

\begin{theorem}
\label{t2}
Let $(\mathcal N,\g)$ be a static globally hyperbolic Lorentzian manifold of dimension $1+n$. Let $\Phi$ be an embedding satisfying \eqref{stationary} with $\eta=0$ and let $\M=\Phi(\R\times M)$ where $M$ is a compact $n$ dimensional submanifold of $N$ with smooth boundary such that Hypothesis~\ref{hypo2} holds. Let $\alpha \in C^{\infty}_c(\M;\mathcal S^{m})$. The following injectivity result holds for the light ray transform on $(\M,\g)$:
$$ \mathcal L_{\beta}\,\alpha =0 \quad \text{for all maximal $\beta$ in $\M$} \implies \exists \,T,U \quad \text{s.t}\quad \alpha \equiv \bar{d}^sT+U\,\bar{g},$$
where $\bar{d}^s$ denotes the symmetrized covariant derivative\footnote{See the expression \eqref{sym_der_def} in Section~\ref{conformal invariance} for the definition.}, $T \in C^{\infty}_c(\M;\mathcal S^{m-1})$, $U \in C^{\infty}_c(\M;\mathcal S^{m-2})$ and $U\,\bar{g}$ denotes the symmetrized tensor product of the tensors $U$ and $\g$.
\end{theorem}  
\bigskip
Let us emphasize that the gauge appearing in the statement of Theorem~\ref{t2} is the natural one since the light ray transform of any tensor of the form $\bar{d}^sT+U\,\bar{g}$ with $T, U$ compactly supported in $\M$, vanishes. We refer the reader to Lemma~\ref{gauge} for the details. Observe also that, akin to Corollary~\ref{cor1}, the result of Theorem~\ref{t2} can be formulated for compactly supported tensor fields on a suitable non-compact Lorentzian manifold. Finally we mention that Theorem~\ref{t2} extends analogous results obtained in \cite[Proposition 1.4]{FIKO}, where only the cases $m=0,1$ were considered. 

\subsection{Previous literature}
The study of injectivity of the light ray transform on tensors of arbitrary rank is motivated in part due to its connection with coefficient determination problems for the wave equation on Lorentzian manifolds from boundary data, as shown for example in \cite{AR,BA,FIKO,KO,Stefanov1,SY,Wa} for the cases $m=0,1$. In the setting of Minkowski spacetime, invertibility of the light ray transform on scalar functions was proved by Stefanov in \cite{Stefanov1}. This was later extended to derive a local inversion result \cite{R}. In \cite{W}, the light ray transform on two tensors was also considered in Minkowski space-time that arises in the study of cosmic strings. There it was showed that the light ray transform recovers space-like singularities and some light like singularities of the two tensor. 

Beyond the Minkowski space time the literature is sparse even in the scalar case $m=0$. Stefanov proved the injectivity of the light ray transform for this case under the geometrical assumptions that the Lorentzian manifold is real analytic and that a convexity type assumption holds \cite{Stefanov2}. In \cite{FIKO}, injectivity of the light ray transform was proved for the cases $m=0,1$ when $(\M,\g)$ is static and the transversal manifold has an injective geodesic ray transform \cite{FIKO}. This result has been generalized to the case of non-smooth scalar functions \cite{FK} and continuous one-forms. Finally, we refer the reader to \cite{LOSU} for the study of the light ray transform on general Lorentzian manifolds. There, it is proven that the space like singularities of a scalar function can be recovered from this map. 

As discussed in Section~\ref{stationarycase}, an immediate corollary of Theorem~\ref{t1} is the unique recovery of a zeroth order coefficient from the Dirichlet to Neumann map for the wave equation in stationary geometries, provided that Hypothesis~\ref{hypo0} or Hypothesis~\ref{hypo1} holds. As far as we know and specifically in the case of Hypothesis~\ref{hypo1}, this is the first instance of a coefficient recovery result for linear hyperbolic partial differential equations in geometries that are not real analytic and do not admit a product structure $\g=-dt^2+g(x)$. In fact the same phenomenon appears in the context of the anisotropic Calder\'{o}n problem (see for example \cite{DKLS}), where the analogous product structure is assumed in the Riemannian context for all known results. Theorem~\ref{t2} provides the generalization of \cite{FIKO} to the more general case of tensors of arbitrary rank $m \geq 2$ in static geometries. We mention that in the case $m=2$, this theorem has applications in transmission ultrasound imaging of moving tissues and organs \cite[Section 5]{LOSU}. It is also related to analysis of the cosmic microwave background radiation \cite{LOSU}.

The analysis in this paper is based on reducing the question of injectivity of the light ray transform in the Lorentzian manifold $(\M,\g)$ to the question of injectivity of a ray transform on the spatial part of the manifold $(M,g_c)$. In the stationary case, the corresponding ray transform is a generalization of the geodesic ray transform, consisting of integrals over a family of curves that solve equations of the type \eqref{generalcurve}. Injectivity of such a ray transform has in fact been studied for a broader family of vector fields $G$ than the specific one given by expression \eqref{G_def}. We refer the reader to \cite{FSU}, and also the appendix section of \cite{UV} written by Hanming Zhou (see also \cite{Z}). In the static case, the corresponding ray transform is the geodesic ray transform on $M$. As discussed in Section~\ref{staticcase}, solenoidal injectivity on tensors of arbitrary rank is known to hold under the assumption of a convex boundary and a global foliation of $M$ by convex hypersurfaces \cite{dHUZ}. 
\subsection{Outline of the paper}
In Section~\ref{prelim}, we begin with the derivation of \eqref{stationary}. We then discuss the natural gauge for the injectivity of the light ray transform and the conformal invariance of this gauge. Section~\ref{thm1section} is concerned with the proof of Theorem~\ref{t1}. Finally, Section~\ref{thm2section} contains the proof of Theorem~\ref{t2}. The latter two sections are independent of each other.

\section{Preliminaries}
\label{prelim}
\subsection{Geometry of stationary Lorentzian manifolds}
\label{stationary_structure}
The aim of this subsection is to construct the canonical embedding $\Phi:\R \times N \to \mathcal N$ corresponding to a Cauchy surface $N$ in $\mathcal N$, such that the metric $\Phi^*\g$ takes the form \eqref{stationary}. Let us denote by $\mathcal E$ the complete Killing vector field on $\mathcal N$, and for each $x \in N$, define $\Phi(\cdot,x)$ as the integral curve 
\begin{align}\label{int_curve} \frac{d}{dt}{\Phi}(t,x)=\mathcal E(\Phi(t,x)),\quad \forall\, t \in \R \quad \text{and}\quad \Phi(0,x)=x.\end{align} 
Existence of a solution $\Phi(t,x)$ for all $t\in \R$ is guaranteed by the completeness of the vector field $\mathcal E$. We will show that $\Phi$ is a diffeomorphism. By global hyperbolicity, any integral curve $\Phi(\cdot,x)$ can not self intersect. As two distinct integral curves can not intersect either, we deduce that $\Phi$ is injective. To see surjectivity, let $y \in \mathcal N$ and consider the integral curve 
$$ \frac{d}{dt}\Psi(t)=\mathcal E(\Psi(t))  \quad \forall\, t \in \R \quad \text{and}\quad \Psi(0)=y.$$
Using the definition of a Cauchy hypersurface and the fact that $\mathcal E$ is time-like, it follows that $\Psi(s) \in N$ for some $s \in \R$. Hence $y=\Phi(-s, \Psi(s))$ and $\Phi$ is surjective. Finally, since $\mathcal E$ is smooth, it follows that $\Phi$ is a diffeomorphism.

Next, we study $\Phi^*\g$. Let $(t,x)$ denote a local coordinate system near a point $p \in \R \times N$ and let $\g_{ij}$ represent the components of the metric in this coordinate system, with $i,j=0,1,\ldots,n$. Since $\mathcal E$ is a Killing vector field, it follows that the components $(\Phi^*\g)_{ij}(t,x)$ are all independent of $t$.
Therefore, we can write: 
$$\Phi^*\g(t,x)= (\Phi^*\g)_{00}(x)\,dt^2 + 2\underbrace{(\Phi^*\g)_{0\alpha}(x)\,dx^\alpha}_{\eta}\,dt+ g(x)$$
where $g=\g|_N$ and the index $\alpha$ runs from $1$ to $n$. Note that $(N,g)$ is a Riemannian manifold since $N$ is a space-like hypersurface, in the sense that all of its tangent vectors are space-like. Since $\label{killing_exp}\partial_t=\Phi^*\mathcal E$, it is easy to see that $\eta=\p_t^\flat=\Phi^*\mathcal E^\flat|_N$, where $\mathcal E^\flat$ is the dual covector associated with $\mathcal E$. Finally, we define
\[ \kappa= |\eta|_{g}^2- (\Phi^*\g)_{00}\]
and observe that $\kappa=|\eta|_g^2-\g(\mathcal E,\mathcal E)>0$.

\subsection{Conformal invariance of the gauge}
\label{conformal invariance}

We begin with a lemma that shows the gauge in Theorem~\ref{t2} is the natural one for the injectivity of the light ray transform on tensors.
\begin{lemma}
\label{gauge}
Let $(\mathcal N,\bar{g})$ denote a globally hyperbolic Lorentzian manifold. Suppose that $T \in C^{\infty}_c(\mathcal N;\mathcal S^{m-1})$ and $U \in C^{\infty}_c(\mathcal N;\mathcal S^{m-2})$. Then:
$$ \mathcal L_\beta (\bar{d}^s T+U\,\bar{g}) =0\quad \text{for all maximal null geodesics $\beta \subset \mathcal N$}.$$
In other words, given $T, U$ as above, $\mathcal L_\beta$ is invariant under the transformation
\begin{align}\label{ng}
\alpha \to \alpha + \bar{d}^sT+U\,\bar{g}.
\end{align}
\end{lemma}
\begin{proof}
Since $\bar{g}(\dot{\beta}(s),\dot{\beta}(s))=0$ along any null geodeisc, it follows trivially that $\mathcal L_\beta (U\,\bar{g})=0$. Now, let us recall the definition of the symmetrized covariant derivative
\begin{multline}
\label{sym_der_def}
[\bar{d}^s T]_{i_1,\ldots,i_m}=\frac{1}{m!}\sum_{\pi \in S(m)}  (\p_{i_{\pi(1)}}T_{i_{\pi(2)},\ldots,i_{\pi(m)}}-\bar{\Gamma}^l_{i_{\pi(2)},i_{\pi(1)}}T_{l,i_{\pi(3)},\ldots,i_{\pi(m)}} \\
-\ldots-\bar{\Gamma}^l_{i_{\pi(m)},i_{\pi(1)}}T_{i_{\pi(2)},\ldots,l}) 
\end{multline}
where $\bar{\Gamma}^i_{jk}$ denotes the Christoffel symbols, $S(m)$ denotes the group of permutations of the set $\{1,\ldots,m\}$ and we are using the Einstein's summation convention with respect to the index $l$. From this identity, together with the geodesic equation
$$ \ddot{\beta}^i(s) + \bar{\Gamma}^i_{jk}(\beta(s)) \dot{\beta}^j(s)\dot{\beta}^k(s)=0$$
it follows that
$$ \p_s T(\beta(s),\dot{\beta}(s))= [\bar{d}^s T]_{i_1\ldots i_m} \dot{\beta}^{i_1}(s)\ldots \dot{\beta}^{i_m}(s)$$
and subsequently we have $\mathcal L_\beta (\bar{d}^sT)=0$ since $T$ is compactly supported.
\end{proof}
Next, we aim to study the light ray transform on tensors under conformal rescalings of the metric and show that the natural gauge for the problem is conformally invariant. We consider a globally hyperbolic Lorentzian manifold $(\mathcal N,\g)$ and use the notation $\mathcal L^{\g}_\beta \alpha$ to emphasize the dependence of the light ray transform on the metric. Let $c>0$ and define $\tilde{g}=c\g$. Using \cite[Section 6, Lemma 6.1]{LOSU1}, we observe that given a maximal null geodesic $\beta:\R\to \mathcal N$ satisfying \eqref{affine} with respect to $\g$ and any non-zero $s_0\in \R$, the same curve $\beta$ parametrized as $\tilde{\beta}(s)=\beta(\sigma(s))$ satisfies \eqref{affine} with respect to $\tilde{g}$, where
$$ \sigma(s)=\,\int_{s_0}^s c(\beta(\tau))^{-1}\,d\tau,$$
with $s \in \R$. This shows that given a $\alpha \in C^{\infty}_c(\mathcal N;\mathcal S^m)$, there holds:
\begin{align}
\label{affine_conf}
 \mathcal L^{\tilde{g}}_{\tilde{\beta}}\tilde{\alpha}=\int_{\R} \tilde{\alpha}(\tilde{\beta}(s),\dot{\tilde{\beta}}(s))\,ds=\int_\R (c^{-m+1}\tilde{\alpha})(\beta(s),\dot\beta(s))\,ds.
\end{align}

Using the above identity, it is clear that the injectivity of the light ray transform on scalar functions is conformally invariant. For tensors of rank $m\geq1$ we have the following lemma that shows the natural gauge for the problem as seen in Theorem~\ref{t2} is conformally invariant as well.

\begin{lemma}
\label{conf_inv_lem}
Let $(\mathcal N,\g)$ be a Lorentzian manifold and consider $\tilde{g}=c\g$ for some smooth positive function $c$. Suppose $T \in C^{\infty}(\mathcal N;\mathcal S^{m-1})$ for some $m=1,\ldots$. There exists $U \in C^{\infty}(\mathcal N;\mathcal S^{m-2})$, such that:
$$ c^{-m+1}\tilde{d}^s T=\bar{d}^s(c^{-m+1}T)+U\,\g.$$  
In the case $m=1$, the tensor $U$ is identically zero.
\end{lemma}

\begin{proof}
We use the notations $\tilde{\Gamma}^k_{ij}$ (resp., $\tilde{d}^s$) and $\bar\Gamma^k_{ij}$ (resp., $\bar{d}^s$) to denote the Christoffel symbols (resp., symmetrized covariant derivative) on $\mathcal N$ with respect to the metrics $\tilde{g}$ and $\g$ respectively. By definition,
\begin{multline}
\label{symder}
[\tilde{d}^s \tilde{T}]_{i_1,\ldots,i_m}=\frac{1}{m!}\sum_{\pi \in S(m)}  (\p_{i_{\pi(1)}}\tilde{T}_{i_{\pi(2)},\ldots,i_{\pi(m)}}-\tilde{\Gamma}^l_{i_{\pi(2)},i_{\pi(1)}}\tilde{T}_{l,i_{\pi(3)},\ldots,i_{\pi(m)}} \\
-\ldots-\tilde{\Gamma}^l_{i_{\pi(m)},i_{\pi(1)}}\tilde{T}_{i_{\pi(2)},\ldots,l}) 
\end{multline}
Next, we define $\phi=-\frac{1}{2}\log c$ and recall the following identity that relates the Christoffel symbols $\bar{\Gamma}^i_{jk}$ and $\tilde{\Gamma}^i_{jk}$ (see \cite[Lemma 6.3]{LOSU1}):
$$ \bar{\Gamma}^i_{jk}=\tilde{\Gamma}^i_{jk} + \delta^i_j\p_k\phi+\delta^i_k\p_j\phi -b^i \tilde{g}_{jk},$$
where $b=\nabla^{\tilde{g}}\phi$.
Using the above identity together with the expression \eqref{symder} we observe that the symmetrized derivative on tensors $\tilde{T} \in C^{\infty}_c(\mathcal N;\mathcal S^{m-1})$ transforms as
\begin{align}\label{sym_eq_2} [\bar{d}^s\tilde{T}]_{i_1,\ldots,i_m}=[\tilde{d}^s\tilde{T}]_{i_1,\ldots,i_m}- \frac{1}{m!}\underbrace{\left[\sum_{\pi \in S(m)}\mathcal S_{\pi}\right]}_{I}-U\tilde{g}\end{align}
where
\begin{align*}
\mathcal S_\pi=&\left((\p_{i_{\pi(1)}}\phi)\, \tilde{T}_{i_{\pi(2)},\ldots,i_{\pi(m)}}+(\p_{i_{\pi(2)}}\phi) \,\tilde{T}_{i_{\pi(1)},i_{\pi(3)},\ldots,i_{\pi(m)}}\right)+\ldots\\
&+\left((\p_{i_{\pi(1)}}\phi)\, \tilde{T}_{i_{\pi(2)},\ldots,i_{\pi(m)}}+(\p_{i_{\pi(m)}}\phi) \,\tilde{T}_{i_{\pi(2)},\ldots,i_{\pi(m-1)},i_{\pi(1)}}\right).
\end{align*}

We can simplify $I$ further, by considering the number of times that a fixed term of the form $(\p_{i_{\tilde{\pi}(1)}}\phi)\,\tilde{T}_{i_{\tilde{\pi}(2)},\ldots,i_{\tilde{\pi}(m)}}$ appears in $I$ with $\tilde{\pi} \in S(m)$. Indeed, observe that:
\begin{multline*}
\sum_{\pi \in S(m)}(\p_{i_{\pi(1)}}\phi)\, \tilde{T}_{i_{\pi(2)},\ldots,i_{\pi(m)}}= \sum_{\pi \in S(m)}(\p_{i_{\pi(2)}}\phi) \,\tilde{T}_{i_{\pi(1)},i_{\pi(3)},\ldots,i_{\pi(m)}}\\
=\ldots=\sum_{\pi \in S(m)}(\p_{i_{\pi(m)}}\phi) \,\tilde{T}_{i_{\pi(2)},\ldots,i_{\pi(1)}}.
\end{multline*}
Consequently, equation \eqref{sym_eq_2} reduces to 
\begin{align}\label{sym_eq_3}
 [\bar{d}^s\tilde{T}]_{i_1,\ldots,i_m}=[\tilde{d}^s\tilde{T}]_{i_1,\ldots,i_m}- \frac{2(m-1)}{m!}\left[\sum_{\pi \in S(m)}  (\p_{i_{\pi(1)}}\phi)\tilde{T}_{i_{\pi(2)},\ldots,i_{\pi(m)}}  \right]-U\tilde{g}.
\end{align}

Next, we consider the tensor $T$ in the statement of the lemma and define $\tilde{T}=c^{1-m}T$. We use the defining expression for the symmetrized derivative \eqref{symder} and the definition of $\phi$ to obtain
\begin{align*}
[c^{-m+1}\tilde{d}^sT]_{i_1,\ldots,i_m}&=[c^{-m+1}\tilde{d}^s(c^{m-1}\tilde{T})]_{i_1,\ldots,i_m}\\
&=[\tilde{d}^s \tilde{T}]_{i_1,\ldots,i_m}-\frac{2(m-1)}{m!}\left[\sum_{\pi \in S(m)}(\p_{i_{\pi(1)}}\phi)\tilde{T}_{i_{\pi(2)},\ldots,i_{\pi(m)}}\right].
\end{align*}
The claim follows from this identity and equation \eqref{sym_eq_3}.
\end{proof}

Combining Lemma~\ref{conf_inv_lem} together with equation \eqref{affine_conf} and Theorems~\ref{t1}$-$\ref{t2} we have the following immediate corollary.

\begin{corollary}
\label{conf_inv}
Given a globally hyperbolic Lorentzian manifold $(\mathcal N,\g)$, injectivity of the light ray transform modulo the gauge \eqref{ng} is conformally invariant. In particular, the injectivity results stated in Theorem~\ref{t1} and Theorem~\ref{t2} hold in the more general setting that $\mathcal N$ is conformally stationary or conformally static respectively.
\end{corollary}


\section{Injectivity of $\mathcal L$ in stationary geometries}
\label{thm1section}
Suppose that $(\M,\g)$ is as in Theorem~\ref{t1}. We are interested in the question of injectivity of the light ray transform. Owing to the conformal invariance of the light ray transform on scalar functions (see Section~\ref{conformal invariance}), we will work with the conformally rescaled metrics $\g_c$ and $g_c$ as discussed in Section~\ref{stationarycase}. For the remainder of this section, we abuse the notation slightly and write $\mathcal L$ to denote the light ray transform on $(\M,\g_c)$ and also identify functions and tensors in $\M$ with their copies in $\Phi^{-1}(\M)$ without explicitly writing the pull-back.

\begin{lemma}
\label{lem_1}
Let $\beta:I \to \M$ be a maximal null geodesic on $(\M,\g_c)$ and write $\beta(s) = (a(s), b(s))$
where $a$ and $b$ are paths on $\R$ and $M$ respectively. Let $T \in \R$. Then $\beta_T:I \to \R$ defined through $\beta_T(s) = (a(s) + T, b(s))$ is a maximal null geodesic on $\M$.
\end{lemma}
\begin{proof}
This follows immediately from the fact that the components of $\g_c(t,x)$ are independent of the time-coordinate $t$.
\end{proof}

Let $f \in C_c^\infty(\M)$ and suppose that $\beta:I \to \mathcal M$ is a maximal null geodesic. Define $\beta_T:I\to \mathcal M$ as translations of $\beta(s)$ along the time coordinate $t$ analogously as above. Then;
\begin{equation}
\label{int_identity}
\begin{aligned}
&\int_\R e^{-\iota\tau T} \mathcal L_{\beta_T}f\,dT
= \int_\R \int_I e^{-\iota\tau T} f(a(s)+T, b(s)) \,dT \,ds
\\&
= \int_I e^{\iota\tau a(s)} \int_\R e^{-\iota\tau r} f(r, b(s)) \,dr \,ds
= \int_I e^{\iota\tau a(s)} \hat f(\tau, b(s)) \,ds.
\end{aligned}
\end{equation}
with $\hat{f}$ denoting the Fourier transform\footnote{We use the notation $\iota$ for the imaginary unit to avoid confusion with the indices.} in $t$. We define the integral transform 
\begin{align}
\label{G_transform}
\mathscr I f(b)= \int_I f(b(s)) \,ds,
\end{align}
where $b = \pi \circ \beta$, with $\pi : \M \to M$ the natural projection and $\beta$ a null geodesic on $\M$.
 
Let us analyze $\mathscr I$ further. Referring to $\beta(s) = (a(s), b(s))$ in Lemma~\ref{lem_1}, we use the shorthand notations
$$
\dot a = \frac{d a}{ds}, \quad
\dot b = \frac{d b}{ds}, \quad |\dot b|_{g_c} = |\dot b|.
$$
As $\dot \beta$ is lightlike, there holds 
$$
-\dot a^2 + 2 \dot{a} \eta_c \dot b + |\dot b|^2 = 0.
$$
Therefore 
\begin{align}
\label{ode1}
\dot a = \eta_c \dot b \pm \sqrt{(\eta_c \dot b)^2 + |\dot b|^2}.  
\end{align}
Let $\bar\Gamma^i_{jk}$ denote the Christoffel symbols on $(\M,\g_c)$ and observe that $\bar\Gamma^i_{00}=0$ for $i=1,\ldots,n$. Using this and the definition of a null geodesic, we see that $b$ satisfies the equation 
\begin{align}
\label{ode2}
\frac{d^2 b^i}{ds^2} 
+ \bar\Gamma^i_{jk}(b(s)) \dot b^j \dot b^k + 2 \bar\Gamma^i_{0k} \dot a \dot b^k = 0, 
\end{align}
for $i=1,\ldots,n$. We can choose, without loss of generality, the positive sign in equation \eqref{ode1}. Indeed, suppose that $(a_{+}(s),b_{+}(s))$ solves \eqref{ode1}--\eqref{ode2} with the the positive sign in \eqref{ode1} for $s \in I$. Then, $(a_-(s),b_-(s)):=(a_{+}(-s),b_+(-s))$ with $s \in -I$ solves the same two equations with the negative sign in \eqref{ode1}. Hence, the choice of sign corresponds to affine re-parametrizations of a fixed null geodesic. For this reason, we will just consider the positive sign in \eqref{ode1}.

Now, equation \eqref{ode2} can be recast in the form 
\begin{align}
\label{b_curves_1}
\nabla^{g_c}_{\dot b} \dot b = G(b, \dot b),
\end{align}
where 
\begin{align}
\label{G_def_1}
G^i(b,\dot b):=(\Gamma^i_{jk} - \bar\Gamma^i_{jk}) \dot b^j \dot b^k - 2 \bar{\Gamma}^i_{0k} (\eta_c \dot b + \sqrt{(\eta_c \dot b)^2 + |\dot b|^2}) \dot b^k,
\end{align}
where $\Gamma^i_{jk}$ denotes the Christoffel symbols on $(M,g_c)$. We will now simplify the latter expression and show that the curves $b \in \mathscr G$ are coordinate invariant in $\M$. To see this, we first observe that
\[
\bar{g}^{-1}_c = \left( \frac{\kappa-|\eta|_{g}^2}{\kappa} \right)\left[\begin{array}{cc} 
-1 & \eta_c^{\sharp} \\
(\eta^{\sharp}_c)^T & \frac{\kappa}{\kappa-|\eta|_{g}^2}\,g_c^{-1}-\eta_c^{\sharp}\otimes\eta_c^{\sharp}
\end{array} \right]
\]
where $\eta_c^{\sharp}$ denotes the canonical vector that is dual to the one-form $\eta_c$ and $^T$ denotes the transposition operation. Now, using the definition of the Christoffel symbols together with the fact that the coefficients of the metric are time-independent, we write
\begin{multline*}
\Gamma^i_{jk}-\bar{\Gamma}^i_{jk}=\underbrace{-\frac{1}{2}(\bar{g}_c)^{i0}\left((\bar{g}_c)_{0,j;k}+(\bar{g}_c)_{0,k;j}\right)}_{I} +\\
\underbrace{\frac{1}{2}((g_c)^{im}-(\bar{g}_c)^{im})\left((g_c)_{mj;k}+(g_c)_{jm;k}-(g_c)_{jk;m}\right)}_{II}
\end{multline*}
where the term $II$ involves a summation over the index $m=1,\ldots,n$. The term $I$ reduces as follows
\[
I= -\frac{1}{2}(\frac{\kappa-|\eta|_{g}^2}{\kappa})(\eta_c^\sharp)^i  \left((\eta_c)_{j;k} + (\eta_c)_{k;j}\right)
\]
Similarly, the term $II$ reduces as follows
\[
\begin{aligned}
II&= \frac{1}{2} ((g_c)^{im}-(\bar{g}_c)^{im})\left((g_c)_{mj;k}+(g_c)_{jm;k}-(g_c)_{jk;m}\right)\\
&= \frac{1}{2} (\frac{\kappa-|\eta|_{g}^2}{\kappa}) (\eta_c^{\sharp})^i (\eta_c^{\sharp})^m \left((g_c)_{mj;k}+(g_c)_{jm;k}-(g_c)_{jk;m}\right)\\
&=\frac{1}{2} (\frac{\kappa-|\eta|_{g}^2}{\kappa}) (\eta_c^{\sharp})^i(\eta_c)_l (g_c)^{ml} \left((g_c)_{mj;k}+(g_c)_{jm;k}-(g_c)_{jk;m}\right)\\
&=(\frac{\kappa-|\eta|_{g}^2}{\kappa}) (\eta_c^{\sharp})^i(\eta_c)_l\Gamma^{l}_{jk}.
\end{aligned}
\]
Combining the expressions for $I$ and $II$ we deduce that
$$ (\Gamma^i_{jk} - \bar\Gamma^i_{jk}) \dot b^j \dot b^k=-(\frac{\kappa-|\eta|_{g}^2}{\kappa})\left((\nabla^{g_c}_{\dot b}\eta_c)\,\dot b\right)\,(\eta_c^{\sharp})^i.$$
We now consider the last term in the expression for $G$. Using the definition of the Christoffel symbols again and the expression of the inverse matrix $\bar g_c^{-1}$ above, this reduces as follows
\[
\begin{aligned}
2\bar{\Gamma}^{i}_{0k}\dot{b}^k&=\bar{g}_c^{im} ((g_c)_{m0;k}-(g_c)_{0k;m})\dot{b}^k=\bar{g}_c^{im} ((\eta_c)_{m;k}-(\eta_c)_{k;m})\dot b^k\\
&=\left(g_c^{im}-\frac{\kappa-|\eta|_{g}^2}{\kappa}(\eta^\sharp_c)^i(\eta^\sharp_c)^m\right) \left((\eta_c)_{m;k}-(\eta_c)_{k;m}\right)\dot b^k.
\end{aligned}
\]
Recalling that $(d\eta_c)_{mk}=(\eta_c)_{m;k}-(\eta_c)_{k;m}$, we conclude that $G$ can be rewritten as given by equation \eqref{G_def}, thus establishing that it is an invariantly defined vector field on $M$. Let us emphasize that the parametrization of the curve $b(s)$ in $M$ with $s \in I$ is not a unit-speed parametrization and is directly induced by the initial choice of an affine parametrization for the null geodesic $\beta$ in $\M$.


\begin{theorem}
\label{I_injectivity}
If $\mathscr I$ is injective then $\mathcal L$ is also injective.
\end{theorem}
\begin{proof}
Suppose that $\mathcal L_\beta f = 0$ where $\beta:I \to M$ denotes any maximal null geodesic in $\mathcal M$ with maximal interval $I$. Differentiating equation \eqref{int_identity} $k$ times with respect to $\tau$ and evaluating at $\tau=0$, we obtain: 
$$0=\sum_{j=0}^k \int_I (\iota a(s))^{k-j}\, \p_\tau^{j}\hat{f}(0,b(s))\,ds\quad \forall\, b \in \mathscr G$$
Setting $k=0$, we have
\begin{align*}\notag 
(\mathscr I \hat f(0,\cdot))(b) = 0\quad \forall\, b\in \mathscr G.
\end{align*}
By injectivity of $\mathscr I$, it holds that $\hat f(0,\cdot) = 0$.
In a similar manner, by using induction on $k$ together with the injectivity of $\mathscr I$, we deduce that
$$\p_\tau^k \hat f(0,\cdot) = 0, \quad \forall\, k \in \N.$$
As $f(t,\cdot)$ is compactly supported in $t$, $\hat f(\tau,\cdot)$ is analytic in $\tau$,
and thus $f$ vanishes everywhere.
\end{proof}

\subsection{Proof of Theorem~\ref{t1}}

It is clear that Theorem~\ref{t1} follows, once we prove injectivity of the ray transform $\mathscr I$ along all maximal curves $b \in \mathscr G$. We prove this under the assumption that Hypothesis~\ref{hypo0} or Hypothesis~\ref{hypo1} holds. In fact, the transform $\mathscr I$ has been studied for more general vector fields $G(z,v)$ than the one given by expression \eqref{G_def} and invertibility is known to hold under some assumptions. When Hypothesis~\ref{hypo1} holds on $(\M,\g)$, injectivity of $\mathscr I$ follows from \cite[Theorem 4.2]{UV} in the appendix by Hanming Zhou and the remarks immediately following that theorem. 

To prove Theorem~\ref{t1} under Hypothesis~\ref{hypo0}, we will use \cite[Theorem 1]{FSU}. There, injectivity of the map $\mathscr I$ is proved under the assumption that the manifold $(M,g_c)$ and the curves in $\mathscr G$ are real analytic and that there are no conjugate points along any such curve. Hence, to conclude the proof, we need to show that if Hypothesis~\ref{hypo0} holds in $\M$, the aforementioned assumptions are satisfied on curves in $\mathscr G$.  

To this end, let us first recall the definition of {\em conjugate points} along curves $b \in \mathscr G$ (following \cite{FSU}) and conjugate points along null geodesics $\beta$ in $\M$. Given any $(s,\xi) \in TM$, we define the exponential map $\widetilde{\exp}_x (s,\xi)=b(s)$ where $b\in \mathscr G$ with $b(0)=x$ and $\dot{b}(0)=\xi$. Subsequently, we say that the point $b(s_0)$ is conjugate to $x$ if $(D_{s,\xi}\widetilde{\exp}_{x})(s_0,\xi_0)$ has rank less than $n$, where $\xi_0=\dot{b}(0)$. The conjugate points on $\M$ are defined analogously, in terms of the exponential map, $\exp:T\M \to T\M$ of the Lorentzian manifold $(\M,\g_c)$ along null geodesics (see for example \cite[Definition 10.9]{On}). 

\begin{lemma}
\label{non-optimal}
If $C^+_N(p)=\emptyset$ for all $p \in \mathcal M$, then there are no conjugate points along any curve $b \in \mathscr G$.
\end{lemma}

\begin{proof}
Suppose for contrary that there exists a curve $b \in \mathscr G$ with a pair of conjugate points $b(0)$ and $b(s_0)$. The above definition of conjugate points implies in particular that there exists a one-parameter family of curves $b_r$ in $\mathscr G$ with $r$ in a small neighborhood of origin, such that 
$$b_r(0)=b(0), \quad \dot{b}_r(0)=\dot{b}(0)+rv$$ for some fixed $v \in T_{b(0)}M$ and 
\begin{align}\label{almost}\dist(b_r(s_0),b(s_0))\leq C r^2\end{align}for some uniform constant $C>0$, where $\dist(\cdot,\cdot)$ is the Riemannian distance function on $(M,g_c)$ (note that $b_0\equiv b$). We define the functions $a_r(s)$ as the solutions to the following differential equation: 
$$\frac{da_r}{ds}= \eta_c \dot b_r \pm \sqrt{(\eta_c \dot b_r)^2 + |\dot b_r|^2}\quad \text{and}\quad a_r(0)=0,$$
where the sign $\pm$ is chosen in order to make the curve $(a_r(s),b_r(s))$ future-pointing. Observe that the curves $\beta_r(s)=(a_r(s),b_r(s))$ define a family of maximal null geodesics in $\mathcal M$ (in what follows, we will drop the subscript $r$ when $r=0$).

Next, we observe that there exists a constant $\epsilon>0$ depending only on $\eta_c$ and $g_c$, such that if 
\begin{align}
\label{aux_const}
|\dist(b_{r_k}(s_0),b(s_0))\,|<\epsilon\, |a_{r_k}(s_0)-a(s_0)|\end{align}
for a sequence $r_k \to 0$, then there exists a causal path between $\beta(s_0)$ and $\beta_{r_k}(s_0)$ for all $k$ sufficiently large. 

If \eqref{aux_const} does not hold for any sequence $r_k \to 0$, then in particular it implies that $|a_{r}(s_0)-a(s_0)|<\frac{C}{\epsilon}r^2$ and all $r$ sufficiently close to zero. But then the first variation of $\beta$ among the family of null geodesics $\beta_r$ must vanish at the point $\beta(s_0)$ and consequently the point $\beta(s_0)$ is a conjugate point to $\beta(0)$ along $\beta$. By \cite[Proposition 10.48]{On}, there exists a future pointing time-like curve connecting $\beta(0)$ to $\beta(s_0)$ and therefore there exists a null cut point on $\beta$ corresponding to $\beta(0)$ which is a contradiction. 

Thus, we assume that \eqref{aux_const} holds, for a sequence $r_k \to 0$ and consequently that there exists a future pointing causal path connecting $\beta(s_0)$ to $\beta_{r_k}(s_0)$, or vise versa, for some $k$. First, we consider the case where this future pointing causal curve is from $\beta(s_0)$ to $\beta_{r_k}(s_0)$. Then the points $\beta(0)$ and $\beta_{r_k}(s_0)$ can be connected through the concatenation of the curve $\beta$ that connects $\beta(0)$ to $\beta(s_0)$ and the causal curve that connects $\beta(s_0)$ to $\beta_{r_k}(s_0)$. By \cite[Proposition 10.46]{On}, we conclude that $\tau(\beta_{r_k}(0),\beta_{r_k}(s_0)) \neq 0$, which implies that $C^+_N(\beta(0))\neq \emptyset$. In the other case that the future pointing causal curve connects $\beta_{r_k}(s_0)$ to $\beta(s_0)$, we can use a similar argument to conclude that $\tau(\beta(0),\beta(s_0)) \neq 0$ and subsequently that $C^+_N(\beta(0)) \neq \emptyset$.     
\end{proof}

We are now ready to prove injectivity of the ray transform $\mathscr I$ when Hypothesis~\ref{hypo0} holds. To apply \cite[Theorem 1]{FSU}, we need to show that curves in $\mathscr G$ are real analytic and that no pair of conjugate points exist on any curve in $\mathscr G$. The latter was shown in Lemma~\ref{non-optimal}. Moreover, since $(\mathcal N,\g)$, $N$ and the Killing vector field $\mathcal E$ are assumed to be analytic, it follows from Section~\ref{stationary_structure} that $\eta_c$ and $g_c$ are real analytic as well and consequently the curves $b \in \mathscr G$ are also real analytic as they solve a second order linear ordinary differential equation \eqref{generalcurve} with real analytic coefficients. 


Before closing the section, we give some examples to demonstrate the convex foliation condition in Hypothesis~\ref{hypo1}.

\begin{example}
\label{example_1}
Let
$$ (\M,\g) \cong (\R\times M, -dt^2+g)$$
with $g$ independent of $t$. First, note that maximal null geodesics $\beta$ in $\M$ can be parameterized as $(t+T;\gamma(t;x,v))$ where $\gamma(t;x,v)$ denotes a unit speed geodesic on $M$ with initial data $(x,v) \in \p_{\text{in}}SM$. This is in complete agreement with equation \eqref{generalcurve}, since the function $G$ vanishes identically in this case. Therefore, the notion of $\mathscr G$ strict convexity as introduced in Section~\ref{stationarycase} coincides with the classical notion of strict convexity in the sense of second fundamental form, and Hypothesis~\ref{hypo1} simply states that $\p M$ is strictly convex and that $(M,g)$ admits a foliation by strictly convex hypersurfaces. 
\end{example}

\begin{example}
\label{example_2}
Let us consider the more general case
$$ (\M,\g) \cong (\R \times M, -dt^2 + \eta \otimes dt + dt \otimes \eta + g),$$
where $(M,g)$ is a Riemannian manifold of dimension $n \geq 3$ with boundary and $\eta$ is a covector field on $M$. Building on the previous example, it is clear that Hypothesis~\ref{hypo1} holds as long as $\p M$ is strictly convex in the classical sense and that $(M,g)$ admits a foliation by strictly convex hypersurfaces, provided that $|\eta|_g$ is sufficiently small in the $C^2(M)$ topology.
\end{example}

\section{Proof of Theorem 2} 
\label{thm2section}
We start by considering an embedding of the form \eqref{stationary} with $\eta \equiv 0$ and satisfying Hypothesis~\ref{hypo2}. Throughout this section and for the sake of brevity of notation we will assume without loss of generality that $\kappa \equiv 1$ so as to discard the notations $\g_c$ and $g_c$ (see Section~\ref{conformal invariance}). Observe that due to the more restrictive form of the metric (compared to the stationary case), null geodesics in $(\M,\g)$ can conveniently be parameterized as
$$\beta(\cdot;r_0,x,v)=(r+r_0;\gamma(r;x,v)),$$ 
with $r_0 \in \R$, $(x,v) \in \p_{\text{in}}SM$ and $\gamma(\cdot;x,v)$ denoting a unit speed geodesic with initial data $(x,v)\in \p_{\text{in}}SM$. 

Owing to this identification of null geodesics, we can recast the light ray transform on $\R \times M$ for $\alpha \in C_c^\infty(\R \times M; \mathcal S^m)$ as 
$$
(\mathcal L \,\alpha)(r_0,x,v) = \int_0^{\tau_+(x,v)} \alpha((r+r_0,\gamma(r;x,v)),(1, \dot \gamma(r;x,v))) dr, 
$$
for all $(r_0,x,v) \in \R \times \p_{\text{in}} SM$. 
\subsection{Notations}
\def\i{\pmb i}
\def\j{\pmb j}
\def\pr{\pmb p}
\def\tfs{\text{tfs}}
\def\t{\text{t}}
 
For symmetric tensors $f$ and $h$,
we denote the symmetrized tensor product simply by $fh$.
In particular, if $f$ and $h$ are 1-forms, then
$$
f h(v,w) = \frac 1 2 (f(v) h(w) + f(w) h(w)), 
\quad v,w \in TM.
$$
Following \cite{DS}, we next define three operators. The operator $$\i : C^\infty(M;S^m) \to C^\infty(M;S^{m+2})$$ is defined through $\i f = fg$, where we recall that $S^m$ denotes the bundle of symmetric tensors of rank $m$ on $M$. Next, the operator $\j$ is the trace with respect to $g$, that is, $$\j : C^\infty(M;S^{m+2}) \to C^\infty(M;S^{m})$$ is the adjoint of $\i$, and in local coordinates we can write,
$(\j f)_{j_1 \dots j_n} = g^{jk} f_{jkj_1 \dots j_n}$.
The composition $\j \i$ is self-adjoint and positive definite \cite[Lem. 2.3]{DS}. In particular, the inverse $(\j \i)^{-1}$ exists. 
Moreover, by the same lemma, the bundle $S^m$ has the orthogonal decomposition into sub-bundles $S^m = \text{Ker}(\j) \oplus \text{Ran} (\i)$. 
Finally, the operator 
$$\pr: C^\infty(M;S^m) \to C^\infty(M;S^{m})$$
is defined to be the orthogonal projection from $S^m$ to $\text{Ker}(\j)$, and it can be written as 
$$
\pr = 1-\i(\j\i)^{-1}\j,
$$
see \cite[Eq. (2.15)]{DS}.

\subsection{Helmholtz decomposition}

Let us first recall the Helmholtz decomposition as proven in \cite[Th. 3.3.2]{S}, that is, given any $\omega \in C^\infty(M; S^m)$, there are unique $\omega^s \in C^\infty(M; S^m)$ and $h \in C^\infty(M; S^{m-1})$ satisfying 
$$
\omega = \omega^s + d^s h, \quad \delta^s \omega^s = 0, \quad h|_{\p M} = 0,
$$
where $\delta^s$ is the adjoint of $d^s$. We say that $\omega$ is {\em solenoidal} if $\omega=\omega^s$. 

For a family $\omega \in C_c^\infty(\R; C^\infty(M; S^m))$ we define $\omega^s(t) = (\omega(t))^s$. 
As the corresponding potential $h(t)$ is obtained by solving the elliptic partial differential equation,
$$
\delta^s d^s h(t) = \delta^s \omega(t), \quad h(t)|_{\p M} = 0,
$$
we see that 
$h \in C_c^\infty(\R; C^\infty(M; S^{m-1}))$ and 
$\omega^s \in C_c^\infty(\R; C^\infty(M; S^m))$.

We define also the Fourier transform in time by 
$$
\widehat \omega(\tau) = \int_\R e^{-\iota \tau t} \omega(t) \,dt.
$$
Then $d^s \widehat \omega(\tau) = \widehat{d^s \omega}(\tau)$
and 
$\delta^s \widehat \omega(\tau) = \widehat{\delta^s \omega}(\tau)$. In particular, $\widehat \omega(\tau) = \widehat{\omega^s}(\tau) + d^s \widehat h(\tau)$
and $\delta^s \widehat{\omega^s}(\tau) = 0$.
As the Helmholtz decomposition of $\widehat \omega(\tau)$ is unique, we obtain 
    \begin{align}\label{helmholtz_fourier}
(\widehat \omega(\tau))^s = \widehat{\omega^s}(\tau). 
    \end{align}

\subsection{Trace-free Helmholtz decomposition}

We will next recall the trace-free Helmholtz decomposition
as discussed for example in \cite{DS}. By \cite[Th. 1.5]{DS}, for any $\omega \in C^\infty(M; S^m)$ there are unique $\omega^\tfs \in C^\infty(M; S^m)$, $h \in C^\infty(M; S^{m-1})$ and $\omega^\t \in C^\infty(M; S^{m-2})$ satisfying 
    \begin{align}\label{tf_helmholtz}
\omega = \omega^\tfs + \i\omega^\t + d^s h, \quad \delta^s \omega^\tfs = 0, \quad h|_{\p M} = 0,
\quad \j \omega^\tfs = 0,
\quad \j h = 0.
    \end{align}
This decomposition is obtained by first solving the following elliptic partial differential equation for $h$,
$$
\delta^s \pr d^s h = \delta^s \pr \omega, \quad h|_{\p M} = 0.
$$
Then $\omega^\t = (\j\i)^{-1} \j (\omega - d^s h)$
and $\omega^\tfs = \omega - \i \omega^\t - d^s h$.

The last equation $\j h = 0$ in the decomposition (\ref{tf_helmholtz}) is in fact a consequence of the first four equations. That is, if
    \begin{align}\label{tf_helmholtz_red}
\omega = \omega_0 + \i \omega_1 + d^s \omega_2, \quad \delta^s \omega_0 = 0, \quad \omega_2|_{\p M} = 0,
\quad \j \omega_0 = 0,
    \end{align}
then $\omega_0 = \omega^\tfs$, $\omega_1 = \omega^\t$ and $\omega_2 = h$.
Indeed, writing $\omega_0' = \omega_0 - \omega^\tfs$, $\omega_1' = \omega_1 - \omega^t$ and $\omega_2' = \omega_2 - h$, we obtain
$\omega_1' = -(\j\i)^{-1}\j d^s \omega_2'$. Then $\pr d^s \omega_2' = - \omega_0'$
and $\omega_2'$ solves 
$$
\delta^s \pr d^s \omega_2' = 0, \quad \omega_2'|_{\p M} = 0.
$$
Therefore $\omega_2' = 0$ and $\omega_2 = h$. Now also $\omega_0 = \omega^\tfs$ and $\omega_1 = \omega^\t$ by \cite[Th. 1.5]{DS}. We record the following consequence that will be useful in what follows. 
\begin{remark}
\label{remark_helmholtz}
{\em If $w=d^s h$ for some $h \in C^{\infty}(M;S^{m-1})$ satisfying $h|_{\p M}=0$, then $w^{\tfs}=0$ and $w^\t=0$.}
\end{remark}

Analogously to the previous section, for a family $\omega \in C_c^\infty(\R; C^\infty(M; S^m))$ we can define 
$$
\omega^\tfs(t) = (\omega(t))^\tfs,
\quad
\omega^\t(t) = (\omega(t))^\t,
$$
that gives smooth families of tensors that are compactly supported in time. 
Observe that $\i \widehat \omega(\tau) = \widehat{\i \omega}(\tau)$ and $\j \widehat \omega(\tau) = \widehat{\j \omega}(\tau)$, and analogously with (\ref{helmholtz_fourier}), we have
$$
(\widehat \omega(\tau))^\tfs = \widehat{\omega^\tfs}(\tau),
\quad 
(\widehat \omega(\tau))^\t = \widehat{\omega^\t}(\tau).
$$

\subsection{Injectivity of the light ray transform on tensors}

For the remainder of the paper and for the sake of brevity, we will abuse the notation slightly and identify tensors in $\M$ with their identical copies in $\Phi^{-1}(\M)$ without explicitly writing the embedding. Let $\alpha \in C^\infty(\M; \mathcal S^m)$
and suppose that $\mathcal L \alpha \equiv 0$. As $\bar g=-dt^2 + g$ and $\alpha$ is symmetric, we write
    \begin{align}\label{alpha_comp_1}
\alpha = f\,dt + \omega+b\,\g,
    \end{align}
where 
$$f \in C_c^\infty(\R; C^\infty(M; S^{m-1}))\quad \omega \in  C_c^\infty(\R; C^\infty(M; S^{m}))\quad b \in \mathcal C^{\infty}_c(\M;\mathcal S^{m-2}).$$

We can simplify \eqref{alpha_comp_1} further by considering the Helmholtz decomposition of $f$, that we denote by $f = f^s + d^sp$. To this end, we begin by writing 
\begin{align*}
d^sp\,dt=\bar{d}^s(p\,dt)+\p_tp\,\bar{g} -\p_tp\,g.
\end{align*}
Note that the term $\bar{d}^s(p\,dt)+\p_tp\,\bar{g}$ takes the form of the gauge \eqref{ng} and by Lemma~\ref{gauge} lies in the kernel of $\mathcal L$, while 
$$-\p_tp\,g \in C_c^\infty(\R; C^\infty(M; S^{m})).$$ 
In particular, we can replace $\omega$ with $\omega - \p_t p\,g$ in \eqref{alpha_comp_1}. As also $b\,\g$ is in the kernel of $\mathcal L$, we can assume without loss of generality that
   \begin{align}\label{alpha_comp}
\alpha = f\,dt + \omega, \quad f=f^s.
    \end{align}

We have the following Fourier slicing lemma

\begin{lemma}\label{lem_Fourier_slice2}
Suppose that $\alpha \in C^\infty_c(\R \times M;S^m)$
is of the form (\ref{alpha_comp}).
Then for $k=0,1,\dots$, and $(x,v) \in \p_{\text{in}} SM$ it holds that 
    \begin{align}\label{Fourier_slice1}
\p_\tau^k \widehat{\mathcal L \alpha}(\tau, x,v)|_{\tau = 0}
&= 
\mathcal I(\p_\tau^k\widehat f(\tau,\cdot)|_{\tau = 0})(x,v) + \sum_{j=0}^{k-1} \binom k j \mathcal R_{k-j}(\p_\tau^j\widehat f(\tau,\cdot)|_{\tau = 0})(x,v)
\\\notag&
+ 
\mathcal I(\p_\tau^k\widehat \omega(\tau,\cdot)|_{\tau = 0})(x,v) + \sum_{j=0}^{k-1} \binom k j \mathcal R_{k-j}(\p_\tau^j\widehat \omega(\tau,\cdot)|_{\tau = 0})(x,v),
    \end{align}
where 
    \begin{align*}
\mathcal R_j \omega(x,v) &=
\int_0^{\tau_+(x,v)} (\iota r)^{j} \omega(\gamma(r;x,v), \dot \gamma(r;x,v)) \,dr,
\quad \omega \in C_c^\infty(M;S^m).
    \end{align*}
\end{lemma}

We are now ready to prove the main theorem.
\begin{proof}[Proof of Theorem~\ref{t2}]
As discussed above, we can write $\alpha$ in the form \eqref{alpha_comp} with $f=f^s$. Now note that for any $(x,v) \in \p_{\text{in}} SM$ we have that $(y,w)\in \p_{\text{in}}SM$ as well, where
$$
y = \gamma(\tau_+(x,v); x,v), 
\quad
w = -\dot\gamma(\tau_+(x,v); x,v).
$$
Moreover, we have that $\mathcal I\omega(x,v) = \mathcal I\omega(y,w)$
for any $\omega \in C_c^\infty(M; S^m)$ with even $m$ but
$\mathcal I\omega(x,v) = -\mathcal I\omega(y,w)$
for any $f \in C_c^\infty(M; S^m)$ with odd $m$. 

Applying (\ref{Fourier_slice1}) with $k=0$ and using the above observation implies that
    \begin{align}\label{k1}
\mathcal I(\widehat f(0)) = 0,
\quad 
\mathcal I(\widehat \omega(0)) = 0.
    \end{align}
Using Hypothesis~\ref{hypo2} together with $f=f^s$ and Remark~\ref{remark_helmholtz} we deduce that 
    \begin{align*}
\widehat{f}(0) = 0, \quad
\widehat{\omega^\tfs}(0) = 0, \quad
\widehat{\omega^\t}(0) = 0.
    \end{align*}

Let us define 
    \begin{align}\label{def_h}
a_{0}(t,x) = \int_{-\infty}^t \omega^\t(t',x) \,dt'.
    \end{align}
As $\omega$ is compactly supported in time, $a_0(t)$ vanishes for $t$ sufficiently small. Moreover, for large $t$,
$$
a_{0}(t) = \int_{-\infty}^\infty \omega^\t(t') dt' = \widehat{\omega^\t}(0) = 0.
$$
Thus $a_{0} \in C_c^\infty(\R; C^\infty(M;S^{m-2}))$.
Observe also that $\p_t a_{0} = \omega^\t$ and hence $\iota \tau \widehat a_{0} = \widehat{\omega^\t}$. In particular,
    \begin{align*}
\p_\tau^k \widehat{\omega^\t}(0) = \iota k\p_\tau^{k-1}\widehat a_{0}(0), 
\quad k=0,1,\dots.
    \end{align*}

In what follows, we will write 
\begin{align*}
\omega = \omega^\tfs + \i\omega^\t + d^s a_1,
\quad 
a_1 = a_1^s + d^s h, 
\end{align*}
to denote the trace-free Helmholtz decomposition of $\omega$
and the Helmholtz decomposition of $a_1$ respectively.
We will use the fact that for any $u \in C_c^\infty(M;S^m)$,
    \begin{align*}
\mathcal R_j(d^s u)
&= 
\iota^j \int_0^{\tau_+} r^j\, d^su (\gamma(r),\, \dot \gamma(r))dr 
= 
\iota^j \int_0^{\tau_+} r^j \p_r (u(\gamma(r),\dot \gamma(r))) dr\\
&= -\iota j \mathcal R_{j-1}(u).
    \end{align*}
When $k=1$, equation (\ref{Fourier_slice1}) reduces to 
the fact that 
    \begin{align*}
&\mathcal I(\p_\tau \widehat{f})
+ \mathcal I(\p_\tau \widehat{\omega^\tfs} + \iota \widehat a_0) + \mathcal R_1(d^s \widehat a_1)
=
\mathcal I(\p_\tau \widehat{f})
+ \mathcal I(\p_\tau \widehat{\omega^\tfs} + \iota \widehat a_0) - \iota\mathcal I(\widehat a_1)
\\&\quad=
\mathcal I(\p_\tau \widehat{f} - \iota\widehat{a_1^s})
+ \mathcal I(\p_\tau \widehat{\omega^\tfs} + \iota \widehat  a_0 g)
    \end{align*}
vanishes at $\tau = 0$.
Note that $\p_\tau \widehat{f} - \iota\widehat{a_1^s}$ is solenoidal and of rank $m-1$. Moreover, the tensor $$w:=\p_\tau \widehat{\omega^\tfs} + \iota \widehat a_0g$$ is of rank $m$ and satisfies 
$$w^{\t}=\iota\widehat a_0g\quad \text{and}\quad w^{\tfs}=\p_\tau \widehat{\omega^\tfs}.$$
Hence at $\tau = 0$,
$$
\p_\tau \widehat{f} = \iota\widehat{a_1^s}, \quad
\p_\tau \widehat{\omega^\tfs} = 0, \quad
\widehat a_0 = 0.
$$

We will now proceed with an an induction argument to show that  
for all $j\in \N$, it holds at $\tau = 0$ that
\begin{align}
\label{ind_hyp}
\p_\tau^j \widehat{f} = \iota j\p_\tau^{j-1}\widehat{a_1^s}, \quad
\p_\tau^j \widehat{\omega^\tfs} = 0, \quad
\p_\tau^{j-1}\widehat a_0 = -\iota (j-1)\p_\tau^{j-2}\widehat h.
\end{align}
Indeed, let us suppose that this hypothesis holds for all $j=1,\ldots,k-1$. Together with (\ref{Fourier_slice1}) this implies that 
    \begin{align*}
&\mathcal I(\p_\tau^{k}\widehat{f}) 
+ 
\iota \sum_{j=0}^{k-1} \binom k j j \mathcal R_{k-j}(\p_\tau^{j-1}\widehat{a_1^s})
+
\mathcal I(\p_\tau^k\widehat{\omega^\tfs} + \iota k\p_\tau^{k-1}\widehat a_0)
\\&\quad
 + \sum_{j=0}^{k-1} \binom k j \mathcal R_{k-j}(\iota j\p_\tau^{j-1} \widehat a_0 + d^s\p_\tau^{j} \widehat a_1)
    \end{align*}
vanishes at $\tau = 0$.
As $a_1$ vanishes on $\R \times \p M$, we have
\begin{equation}\label{random_eq_2}    
\begin{aligned}
\mathcal R_{k-j}(d^s\p_\tau^{j} \widehat a_1)
&= -\iota (k-j) \mathcal R_{k-(j+1)}(\p_\tau^{j} \widehat a_1)\\
&= -\iota (k-j) \mathcal R_{k-(j+1)}(\p_\tau^{j} \widehat a_1^s) -\iota (k-j) \mathcal R_{k-(j+1)}(d^s\p_\tau^{j} \widehat h).
    \end{aligned}
\end{equation}
Next, using the identity
    \begin{align*}
-\iota \sum_{j=0}^{k-1}  \frac{k!}{j! (k-j)!} (k-j) 
\mathcal R_{k-(j+1)}(\p_\tau^{j} \widehat{a_1^s})
= 
-\iota \sum_{j=1}^{k} \binom k j j \mathcal R_{k-j}(\p_\tau^{j-1} \widehat{a_1^s}),
    \end{align*}
together with \eqref{random_eq_2}, we see that 
    \begin{align*}
&\mathcal I(\p_\tau^{k}\widehat{f}) 
+
\mathcal I(\p_\tau^k\widehat{\omega^\tfs} + \iota k\p_\tau^{k-1}\widehat a_0)
- \iota k \mathcal I (\p_\tau^{k-1} \widehat{a_1^s})
\\&\quad
 + \sum_{j=0}^{k-1} \binom k j \mathcal R_{k-j}(\iota j\p_\tau^{j-1} \widehat a_0)
 -\iota \sum_{j=0}^{k-1} \frac{k!}{j! (k-(j+1))!} \mathcal R_{k-(j+1)}(d^s \p_\tau^{j} \widehat h)
    \end{align*}
vanishes at $\tau = 0$.
As $h$ vanishes on $\R \times \p M$, we have for $j = 0, \dots, k - 2$,
    \begin{align*}
\mathcal R_{k-(j+1)}(d^s\p_\tau^{j} \widehat h)
= -\iota(k-(j+1)) \mathcal R_{k-(j+2)}(\p_\tau^{j} \widehat h),
    \end{align*}
and for $j=k-1$,
    \begin{align*}
\mathcal R_{k-(j+1)}(d^s\p_\tau^{j} \widehat h)
= \mathcal I(d^s\p_\tau^{j} \widehat h) = 0.
    \end{align*}
We rewrite 
    \begin{align*}
&-\iota \sum_{j=0}^{k-1} \frac{k!}{j! (k-(j+1))!} \mathcal R_{k-(j+1)}(d^s \p_\tau^{j} \widehat h)
= 
\sum_{j=0}^{k-2} \frac{k!}{j! (k-(j+2))!}
\mathcal R_{k-(j+2)}(\p_\tau^{j} \widehat h)
\\&\quad
=
\sum_{j=2}^{k} \frac{k!}{(j-2)! (k-j)!}
\mathcal R_{k-j}(\p_\tau^{j-2} \widehat h),
    \end{align*}
and, using $\p_\tau^{j-1}\widehat a_0 = -\iota (j-1)\p_\tau^{j-2}\widehat h$,
    \begin{align*}
\sum_{j=0}^{k-1} \binom k j \mathcal R_{k-j}(\iota j\p_\tau^{j-1} \widehat a_0)
= \sum_{j=2}^{k-1} \frac {k!}{(j-2)!(k-j)!} \mathcal R_{k-j}(\p_\tau^{j-2}\widehat h).
    \end{align*}
Therefore
    \begin{align*}
&\mathcal I(\p_\tau^{k}\widehat{f}) 
+
\mathcal I(\p_\tau^k\widehat{\omega^\tfs} + \iota k\p_\tau^{k-1}\widehat a_0)
- \iota k \mathcal I (\p_\tau^{k-1} \widehat{a_1^s})
+ k(k-1) \mathcal I(\p_\tau^{k-2} \widehat h)
    \end{align*}
vanishes at $\tau = 0$. We obtain at $\tau = 0$,
    \begin{align*}
\p_\tau^{k}\widehat{f} = \iota k \p_\tau^{k-1} \widehat{a_1^s},
\quad 
\p_\tau^k\widehat{\omega^\tfs} = 0,
\quad \iota \p_\tau^{k-1}\widehat a_0 = (k-1) \p_\tau^{k-2} \widehat h,
    \end{align*}
and this closes the induction argument.

We can now use \eqref{ind_hyp} to deduce that 
$$ f =\p_t a_1^s,\quad \omega^\tfs = 0,\quad a_0=-\p_t h.$$ 
To see this, recall that since the functions $f$, $a_1^s$, $\omega^\tfs$, $a_0$ and $h$ are compactly supported in time, their Fourier transforms in $t$ are real analytic. Hence:
$$\hat{f}(\tau,\cdot)= \sum_{k=0}^{\infty} \p^k_\tau \hat f(0,\cdot) \frac{\tau^k}{k!}=\iota\tau\sum_{k=0}^{\infty} \p^k_\tau \widehat{a_1^s}(0,\cdot) \frac{\tau^k}{k!}$$
implying that $f=\p_t a_1^s$. The other two claims follow similarly. Recalling also that $\omega^\t=\p_t a_0$, equation \eqref{alpha_comp} can be rewritten as
$$\alpha = \p_t a_1^s\, dt + d^s a_1 + \p_t a_0 \,g.$$ 
This expression can be further simplified to obtain 
\begin{align}
\label{eq_f}
\alpha = \bar{d}^s\underbrace{(a_0\,dt+a_1)}_{T}+\underbrace{(\p_ta_0)}_{U}\,\bar{g},
\end{align}
thus concluding the proof of the theorem.
\end{proof}    


\medskip\paragraph{\bf Acknowledgements}
A.F was supported by EPSRC grant EP/P01593X/1, J.I. was supported by the Academy of Finland (decision 295853) and L.O was supported by the EPSRC grants EP/P01593X/1 and EP/R002207/1. The authors thank Steve Zelditch for the suggestion to study the
light ray transform on stationary globally hyperbolic Lorentzian manifolds.

\bibliographystyle{abbrv}

\ifoptionfinal{}{
}
\end{document}